\documentclass[psamsfonts]{amsart}
\usepackage{geometry}
\usepackage[utf8]{inputenc}
\usepackage{amsmath} 
\usepackage{amssymb}
\usepackage{amsthm} 
\usepackage{mathtools}
\usepackage[shortlabels]{enumitem}
\usepackage[dvipsnames]{xcolor}
\usepackage{tensor} 
\usepackage{comment} 
\usepackage{relsize}
\usepackage[english]{babel}
\usepackage{mathrsfs}
\usepackage{xypic}
\usepackage{hyperref} 
\usepackage{physics}
\usepackage{bbm}
\usepackage{csquotes}
\usepackage[
backend=biber,
style=alphabetic
]{biblatex}
\newcommand{\pa}[2]{\frac{\partial #1 }{\partial #2}}

\newcommand{\supp}[1]{\mathrm{supp}\left(#1\right)}
\newcommand{\aA}{{{\scriptstyle{}^{\scriptstyle a}}\!\!\mathcal{A}}}
\newcommand{\cj}[1]{\overline{#1}}

\DeclareMathOperator{\C}{\mathbb{C}}
\DeclareMathOperator{\R}{\mathbb{R}}

\DeclareMathOperator{\N}{\mathbb{N}}

\DeclareMathOperator{\F}{\mathbb{F}}

\DeclareMathOperator{\vL}{\mathcal{L}}
\DeclareMathOperator{\vR}{\mathcal{R}}

\DeclareMathOperator{\ad}{ad}

\DeclareMathOperator{\divides}{ \, \big|}

\DeclareMathOperator{\sgn}{\mathrm{sgn}}
\DeclareMathOperator{\foreach}{\mathrm{ \ \forall }}

\DeclareMathOperator{\inv}{^{-1}}

\newtheorem{theorem}{Theorem}[section]
\newtheorem{proposition}[theorem]{Proposition}
\newtheorem{corollary}[theorem]{Corollary}
\newtheorem{lemma}[theorem]{Lemma}
\newtheorem{remark}[theorem]{Remark}
\newtheorem{manualtheoreminner}{Theorem}
\newenvironment{manualtheorem}[1]{%
  \renewcommand\themanualtheoreminner{#1}%
  \manualtheoreminner
}{\endmanualtheoreminner}

\theoremstyle{definition}
\newtheorem{definition}{Definition}[section]

\theoremstyle{definition}

  {\begin{list}{}%
        {\setlength{\leftmargin}{#1}}%
        \item[]%
  }
  {\end{list}}

\usepackage{dsfont}

\makeatletter
\@ifpackageloaded{txfonts}\@tempswafalse\@tempswatrue
\if@tempswa
  \DeclareFontFamily{U}{txsymbols}{}
  \DeclareFontFamily{U}{txAMSb}{}
  \DeclareSymbolFont{txsymbols}{OMS}{txsy}{m}{n}
  \SetSymbolFont{txsymbols}{bold}{OMS}{txsy}{bx}{n}
  \DeclareFontSubstitution{OMS}{txsy}{m}{n}
  \DeclareSymbolFont{txAMSb}{U}{txsyb}{m}{n}
  \SetSymbolFont{txAMSb}{bold}{U}{txsyb}{bx}{n}
  \DeclareFontSubstitution{U}{txsyb}{m}{n}
  \DeclareMathSymbol{\aleph}{\mathord}{txsymbols}{64}
  \DeclareMathSymbol{\beth}{\mathord}{txAMSb}{105}
  \DeclareMathSymbol{\gimel}{\mathord}{txAMSb}{106}
  \DeclareMathSymbol{\daleth}{\mathord}{txAMSb}{107}
\fi
\makeatother

\title{Invariant Geometric Structures on Complex Almost Abelian Lie Groups}

\author[Z. Avetisyan]{Zhirary Avetisyan}
\author[A. Brauer]{Abigail Brauer}
\author[O. Buran]{Oderico-Benjamin Buran}
\author[J. Morentin]{Jimmy Morentin}
\author[T. Wang]{Tianyi Wang}

\begin{document}

\begin{abstract}
An almost Abelian group is a non-Abelian Lie group with a codimension 1 Abelian subgroup. This paper investigates invariant Hermitian and K\"{a}hler structures on connected complex almost Abelian groups. We find explicit formulas for the left and right Haar measures, the modular function, and left and right generator vector fields on simply connected complex almost Abelian groups. From the generator fields, we obtain invariant vector and tensor field frames, allowing us to find an explicit form for all invariant tensor fields. Namely, all such invariant tensor fields have constant coefficients in the invariant frame. From this, we classify all invariant Hermitian forms on complex simply connected almost Abelian groups, and we prove the nonexistence of invariant K\"{a}hler forms on all such groups. Via constructions involving the pullback of the quotient map, we extend the explicit description of invariant Hermitian structures and the nonexistence of K\"{a}hler structures to connected almost Abelian groups.
\end{abstract}

\maketitle

\section{Introduction}
An almost Abelian Lie algebra is a non-Abelian Lie algebra with a codimension one Abelian ideal. In analogy, an almost Abelian Lie group is a Lie group with a codimension 1 Abelian subgroup; almost Abelian Lie groups have almost Abelian Lie algebras. The present paper only considers complex almost Abelian Lie groups and algebras. 

Almost Abelian groups include important classes of groups including groups having Lie algebras of type Bi(II)-Bi(VII) and the Heisenberg group. Almost Abelian groups also arise in cosmological models, dynamical systems, integrable systems, PDEs, and solvmanifolds. This last application is of particular relevance to this paper, and, furthermore, almost Abelian solvmanifolds have seen extensive study in recent years (in particular complex almost Abelian solvmanifolds); see: \cite{Andriot_2011}, \cite{Andrada_2017}, \cite{freibert2011cocalibrated}, \cite{fino2022balanced}, \cite{Fino_2021}, and \cite{Stanfield_2021}.

A solvmanifold is a quotient $G/H$ of a simply connected solvable Lie group $G$ and a discrete subgroup $H$. Since (complex) connected almost Abelian Lie groups are solvable, we can consider almost Abelian solvmanifolds. This paper studies invariant structures on complex simply connected almost Abelian groups and generalizes them to connected almost Abelian groups (being quotients of the simply connected groups by discrete subgroups), and we hope to eventually generalize these results to almost Abelian solvmanifolds.

The following results are found in this paper. Let $G$ be a simply connected complex almost Abelian group. In Prop \ref{HaarMeasureProp}, the explicit formula for the left and right Haar measures and the associated modular function are found. In Corollary \ref{GenVecFields} we find explicit forms for the left and right generator fields of $G$, from which we obtain explicit invariant tensor and vector field frames for $G$. In Prop \ref{ConstantCoeffTensorFields} it is shown that all invariant tensor fields on $G$ have constant coefficients in the invariant frame. In particular, all invariant Hermitian forms on $G$ have constant coefficients in the invariant frame (Corollary \ref{InvariantHermitianMetricsHaveConstantCoeffs}). These results are used to prove our main theorem:
\begin{theorem} \label{NonexistKahlerSimpConn}
There are no left-invariant K\"{a}hler metrics on a simply connected almost Abelian group. 
\end{theorem}
Lastly, these results are generalized in \S\ref{B4} to complex connected almost Abelian groups.

\section{Preliminaries and Basic Properties}\label{B2}

We briefly recall some fundamentals of (complex) almost Abelian Lie algebras, and we refer the reader to \cite{Ave18} for more details. An almost Abelian Lie algebra is a Lie algebra containing a codimension 1 Abelian subalgebra. We use the formal device of an $\N$-graded multiplicity function to completely capture the data of a finite dimensional complex almost Abelian Lie algebra.

Let $\F$ be a field, and let $\sigma_{\F}\subset \F[X]$ denote the set of monic irreducible polynomials with coefficients in $\F$. A \textit{$\N$-graded multiplicity function} is a map $\aleph:\sigma_{\F}\times \N\to \mathcal{C}$, where $\mathcal{C}$ denote the class of cardinals. In the case where $\F=\C$, every monic irreducible polynomial $p_\mu\in \sigma_{\C}$ is of the form $p_\mu(X)=X-\mu$ where $\mu\in \C$. Therefore we can identify $\sigma_{\C}$ with $\C$ and view $\aleph:\C\times \N\to \mathcal{C}$.

We define 
$$\supp{\aleph}=\qty{p\in \C[X]: \exists n\in \N\text{ such that }\aleph(p,n)\neq 0}.$$
Then, for a polynomial $p_\mu\in \sigma_{\C}$, we define $J(p,n)=\mu \mathds{1}+N_n$, where $N_n\in M_{n\times n}(\C)$ with 1's on the super diagonal and zeros everywhere else. Then we may define
$$
J(\aleph)=\bigoplus_{p\in \supp{\aleph}}\bigoplus_{n=1}^\infty\bigoplus_{\aleph(p,n)}J(p,n).
$$

It is shown in \cite{Ave18} that a finitely supported $\N$-graded multiplicity function $\aleph$ uniquely and completely determines the structure of a complex almost Abelian Lie algebra by determining the Jordan matrix $J(\aleph)$ which serves as a matrix representation for $\ad_{e_0}\in \text{End}(\C^d)$, where $e_0=(0,1)\in \C^d\rtimes \C$.

Thus, the data of a finitely supported $\N$-graded multiplicity function is the same as the data of a finite-dimensional almost Abelian Lie algebra. Namely, a finitely supported multiplicity function $\aleph$ is the same as the data of the almost Abelian Lie algebra $\aA(\aleph)=\mathbf{V}\rtimes_{\ad_{e_0}} \C e_0$, where $\ad_{e_0}=J(\aleph)$ and $\mathbf{V}=\C^{\dim_{\C}(\aleph)}$.

Having established the formal apparatus of multiplicity functions, we recall Proposition 3.3 from \cite{avetisyan2023topological} that the simply connected representation of a complex almost Abelian group with finitely supported multiplicity function $\aleph$ is given by:
\begin{proposition} \label{SimpConnRep}
For a finitely supported multiplicity function $\aleph$, let
\begin{equation*}
    G \coloneqq \left\{ 
    \begin{pmatrix}
    1 & 0 & 0 \\
    v & e^{tJ(\aleph)} & 0 \\
    t & 0 & 1
    \end{pmatrix} 
    \: \middle| \ (v,t) \in \C^d \oplus \C \right\}
\end{equation*}
Then $G$ is a complex simply connected Lie group with Lie algebra isomorphic to $\aA(\aleph)$.
\end{proposition}

In order to abbreviate our notations, we will use the notation $(v,t)$ for an element of a complex almost Abelian Lie algebra $\aA(\aleph) = \mathbf{V} \rtimes \C$, and we will use the notation $[v,t]$ to denote the matrix
\begin{equation*}
    \begin{pmatrix}
    1 & 0 & 0 \\
    v & e^{tJ(\aleph)} & 0 \\
    t & 0 & 1
    \end{pmatrix}
\end{equation*}
from the simply connected representation of Prop. \ref{SimpConnRep}; i.e., $[v,t]$ is an element of the simply connected group with Lie algebra $\aA(\aleph)$. Using these convenient representations, we recall Remark 3.5 from \cite{avetisyan2023topological}:

\begin{remark} \label{ExpIsId}
Let $G$ be the simply connected group that has Lie algebra $\aA(\aleph)$. It follows that on the Abelian Lie subalgebra $\ker(J(\aleph)) \oplus \C$ the exponential map $\exp: \aA(\aleph) \to G$ associated with $G$ is given by:
\begin{equation*}
    \exp((v,t)) = [v,t], \qquad \forall (v,t) \in \ker(J(\aleph)) \oplus \C
\end{equation*}
\end{remark}

We recall the following standard fact from Lie theory:

\begin{lemma} \label{ExpMapToCenter}
Let $\mathfrak{g}$ be an arbitrary Lie algebra, $G$ be a connected matrix Lie group that has Lie algebra $\mathfrak{g}$, and let $\exp_G: \mathfrak{g} \to G$ be the corresponding exponential map (specific to $G$). Then $\exp_G(Z(\mathfrak{g})) \subseteq Z(G)$.
\end{lemma}

We will also need a result on the description of the center of a complex simply connected almost Abelian group. This is Proposition 4.2 from \cite{avetisyan2023topological}.

\begin{proposition} \label{CenterSimpConn}
The center of the simply connected almost Abelian Lie group $G$ with Lie algebra $\aA(\aleph)$ is given by
\begin{align*}
    Z(G) &= \exp_{G}[Z(\aA(\aleph))] \times T_\aleph = \exp_{G}[Z(\aA(\aleph)) \times T_\aleph] \\
    &=\{(u,s) \in \C^d \rtimes \C \divides u \in \ker(J(\aleph)),\ e^{s J(\aleph)} = \mathds{1} \}
\end{align*}
where $\exp_{G}: \aA(\aleph) \to G$ is the associated exponential map with $G$.

Also, the preimage under the exponential map (associated with $G$) of the identity component of the center is:
\begin{equation*}
    \exp_{G}\inv[Z(G)_0] = Z(\aA(\aleph))
\end{equation*}
\end{proposition}


\section{Invariant Measures and Frames}
We now find the left and right Haar measures $\mu_L$ and $\mu_R$, and the modular function $\Delta$, on a simply connected almost Abelian group $G$.
\begin{proposition} \label{HaarMeasureProp}
Denote the left Haar measure by $\mu_L$, the right Haar measure by $\mu_R$, and the modular function associated to the Haar measures by $\Delta$. Then:
\begin{equation*}
    d\mu_L(v,t) = e^{-2 \tr[\Re(tJ(\aleph))]} dv\,dt, \quad d\mu_R(v,t) = dv \, dt, \quad \Delta(v,t) = e^{-2 \tr[\Re(tJ(\aleph))]}.
\end{equation*}

\begin{proof}
In order to check the left and right Haar measures, we simply need to check that they are indeed invariant under left and right multiplication, respectively. So, fix an arbitrary $[u,s] \in G$. First consider left multiplication:
\begin{equation} \label{leftmult}
    \Phi_{[u,s]}([v,t]) := [u,s][v,t] = [u + e^{sJ(\aleph)} v, t + s] =: [v', t'].
\end{equation}

The Jacobian matrix of the coordinate transformation $(v,t) \mapsto (v',t')$, denoted here by $\mathcal{J}(v',t';v,t)$, can be seen from \eqref{leftmult} to be:
\begin{equation*}
    \mathcal{J}_{\C}(v',t';v,t) \coloneqq \begin{pmatrix}
    \pdv{v'}{v} & \pdv{v'}{t} \\
    \pdv{t'}{v} & \pdv{t'}{t}
    \end{pmatrix}
    = \begin{pmatrix}
    e^{sJ(\aleph)} & 0 \\
    0 & 1
    \end{pmatrix},
\end{equation*}
when $v', t', v, t$ are considered as complex vectors. But we are interested in transforming the Lebesgue measure. Let $\mathcal{J}_{\R}$ represent the corresponding real Jacobian. We use the known (see \cite{range1998holomorphic}, pg. 19) equality $\det(\mathcal{J}_{\R}) = \det(\mathcal{J}_{\C} \mathcal{J}_{\C}^*)$. Thus
\begin{equation*}
    \det\big(\mathcal{J}_{\R}\big) = \det\big( (e^{sJ(\aleph)}) (e^{sJ(\aleph)})^* \big) = e^{2 \tr[\Re (sJ(\aleph))]}.
\end{equation*}

Then we calculate for our candidate of the left Haar measure:
\begin{align*}
    d\mu_L\big(\Phi_{[u,s]}([v,t])\big) &= d\mu_L([v',t']) \\
    &= e^{- 2 \tr[\Re(t'J(\aleph))]} dv' \, dt' \\
    &= e^{-2 \tr[\Re(t'J(\aleph))]} \abs{\mathcal{J}_{\R}(v',t';v,t)} dv \, dt \\
    &= e^{- 2 \tr[\Re((t+s)J(\aleph))]} e^{2 \tr[\Re(sJ(\aleph))]} dv\, dt \\
    &= e^{-2 \tr[\Re(tJ(\aleph))]} dv\, dt \\
    &= d\mu_L([v,t]).
\end{align*}

Now consider right multiplication for a fixed arbitrary $[u,s] \in G$:
\begin{equation*}
    \Psi_{[u,s]}([v,t]) := [v,t][u,s] = [v + e^{tJ(\aleph)}u, t + s] =: [v'', t''].
\end{equation*}
Then
\begin{align*}
    d\mu_R\big( \Psi_{[u,s]}([v,t]) \big) &= d\mu_R([v'', t'']) \\
    &= dv''\, dt'' \\
    &= \abs{\mathcal{J}_{\R}(v'', t'', v,t)} dv\, dt \\
    &=  dv\, dt \\
    &= d\mu_R([v,t]),
\end{align*}
where the complex Jacobian matrix $\mathcal{J}_{\C}(v'', t''; v,t)$ is calculated to be:
\begin{equation*}
    \mathcal{J}_{\C}(v'', t''; v,t) = \begin{pmatrix}
    1 & J(\aleph) e^{tJ(\aleph)}u \\
    0 & 1
    \end{pmatrix},
\end{equation*}
from which it is clear the real Jacobian is also 1.

Lastly, we calculate that the modular function is
\begin{equation*} \label{ModFuncEq}
    \Delta(v,t) = \frac{d\mu_L(v,t)}{d \mu_R(v,t)} = e^{-2 \tr[\Re\big(tJ(\aleph)\big)]}.
\end{equation*}
\end{proof}
\end{proposition}

Let $G$ be a simply connected Lie group with Lie algebra $\aA(\aleph)$. Recall that the value of a vector field $X \in C^\infty(TG)$ evaluated at a point $[v,t] \in G$ can be written in terms of the coordinate basis vectors:
$$ X_{[v,t]} = X^i([v,t]) \pa{}{x^i} \bigg |_{[v,t]}. $$

Working with the global chart $(v,t) \in \C^d \rtimes \C = G$ as in Prop. \ref{SimpConnRep}, a vector $X \in T_{[v,t]}G$ in these coordinates will take the form of a $d+1$ dimensional row such that 
$$(u, s) \mapsto  \begin{pmatrix}
u^\top & s
\end{pmatrix} = \begin{pmatrix}
X^1([v,t]) & X^2([v,t]) & \dots & X^{d+1}([v,t])
\end{pmatrix}.$$
We sometimes will also identify $X$ directly with the $d+1$ dimensional row $\begin{pmatrix}
u^\top & s
\end{pmatrix}$.

\begin{proposition} \label{GenVecFields}
For every tangent vector $X \in T_{[0,0]}G$ at the identity, 
the corresponding left and right generator vector fields are given by 
$$\vL_X (v,t) = X \begin{pmatrix}
\mathds{1} & 0 \\ v ^\top J(\aleph)^\top  & 1
\end{pmatrix}, \qquad \vR_X (v, t) = X \begin{pmatrix}
e^{tJ(\aleph)^\top} & 0 \\
0 & 1 
\end{pmatrix}, \qquad (v,t) \in G.$$
\begin{proof}
Let $ (-1, 1) \ni \tau \mapsto (u(\tau), s(\tau)) \in G$ be a smooth curve with $(u(0), s(0)) = (0,0)$ such that $(u'(0), s'(0)) = X$. Consider the smooth 1-parameter family of biholomorphisms
$$[v,t] \mapsto [u(\tau), s(\tau)] [v,t] = [u(\tau) + e^{s(\tau)J(\aleph)}v, s(\tau) + t] = [v(\tau), t(\tau)], \quad \forall (v,t) \in G, \quad \tau \in (-1, 1).$$
Then, 
$$  \vL_X(v,t) = \frac{d}{d\tau}[v(\tau), t(\tau)] \bigg |_{\tau = 0} = \frac{d}{d\tau} [u(\tau) + e^{s(\tau)J(\aleph)}v, s(\tau) + t] \bigg |_{\tau = 0} = (u'(0) + s'(0)J(\aleph)v, s'(0)) $$ $$  = \begin{pmatrix}
    {(u'(0) + s'(0)J( \aleph)v)}^\top & s'(0) \\
    \end{pmatrix} = \begin{pmatrix}
    {u'(0)}^\top & {s'(0)}
    \end{pmatrix} 
    \begin{pmatrix}
    \mathds{1} & 0  \\ v^\top J(\aleph)^\top & 1
    \end{pmatrix} =  X \begin{pmatrix} \mathds{1} & 0 \\ v ^\top J(\aleph)^\top  & 1
\end{pmatrix},
$$
as desired. Next consider the smooth 1-parameter family of biholomorphisms given by 
$$[v,t] \mapsto [v,t] [u(\tau), s(\tau)] = [v + e^{tJ(\aleph)}u(\tau), t + s(\tau)] = [v(\tau), t(\tau)], \quad \forall (v,t) \in G, \quad \tau \in (-1, 1).$$
Then, 
$$  \vR_X(v,t) = \frac{d}{d\tau}[v(\tau), t(\tau)] \bigg |_{\tau = 0} = \frac{d}{d\tau} [v + e^{tJ(\aleph)}u(\tau), t + s(\tau)] \bigg |_{\tau = 0} = (e^{tJ(\aleph)}u'(0), s'(0)) $$ 
$$ = \begin{pmatrix}
    {(e^{tJ(\aleph)}u'(0))}^\top & s'(0) \\
    \end{pmatrix} = \begin{pmatrix}
    {u'(0)}^\top & {s'(0)}
    \end{pmatrix} 
    \begin{pmatrix}
    {e^{tJ(\aleph)}}^\top & 0  \\ 0 & 1
    \end{pmatrix} =  X \begin{pmatrix}
    {e^{tJ(\aleph)}}^\top & 0  \\ 0 & 1
    \end{pmatrix}.
$$
Again, as desired. 
\end{proof}
\end{proposition}


\begin{corollary} \label{frameOfHoloVF}
    A global frame of left- and right- invariant vector fields on the simply connected almost Abelian Lie group $G$ as above can be formed by columns of the matrices
    \begin{equation*}
        \begin{pmatrix}
        e^{tJ(\aleph)}&0\\
        0&1
        \end{pmatrix}
        \qquad\text{ and }\qquad
        \begin{pmatrix}
        \mathds{1} & J(\aleph)v\\
        0&1
 \end{pmatrix}
 \end{equation*}
respectively. A global frame of left- and right- invariant 1-forms on the simply connected almost Abelian Lie group $G$ as above can be formed by rows of the matrices
    \begin{equation*}
 \begin{pmatrix}
 e^{-tJ(\aleph)}&0\\
 0&1
 \end{pmatrix}
 \qquad\text{ and }\qquad
 \begin{pmatrix}
 \mathds{1} & -J(\aleph)v\\
 0&1
 \end{pmatrix}
 \end{equation*}
respectively.
\end{corollary}

\begin{proof}
The first statement follows from the expression of $\mathcal{L}_X$ and $\mathcal{R}_X$ in the previous proposition. Since the right generator vector fields are precisely left-invariant vector fields, finding a global frame of left-invariant vector fields is equivalent to finding a basis of the right generator vector fields. Hence by letting $X$ run over a standard basis of $T_{[0,0]}G$, which we identify as a standard basis of $\C^{d+1}$, we get the first statement.
 
For the second statement, define the dual frame $\{X^i\}_{i=1}^{d+1}$ of $\{X_i\}_{i=1}^{d+1}$ by $X^i(X_j)=\delta_j^i$. Then clearly, if $X_i$ is the columns of a matrix, then $X^i$ are the rows of the inverse matrix. This concludes the proof.
\end{proof}


\begin{remark} \label{frameOfAntiVF}
Let $\{X^j\}_{j=1}^{d+1}$ be the above left-invariant frame of holomorphic covector fields, where $X^j = X_a^j dz^a$ and each $X_a^j$ is holomorphic for each $a$ and $j$. Define $\{\overline{X}^j\}_{j=1}^{d+1}$, where $\overline{X}^j = \overline{X}_a^j d\overline{z}^a$.  Then $\{\overline{X}^j\}_{j=1}^{d+1}$ is a frame of antiholomorphic covector fields.
\begin{proof}
Let $\overline{W} := \overline{W_j X^j}$ be an antiholomorphic covector field. Let $\{X_j\}_{j=1}^{d+1}$ be the left invariant holomorphic vector field frame we found in Corollary \ref{frameOfHoloVF}, and let $\{X^j\}_{j=1}^{d+1}$ be the corresponding holomorphic covector field frame. For each $1 \leq j \leq d+1$ let $X^j = X_a^j dz^a$, and recall that $X^j(X_k) = \delta_k^j$. Let $p \in G$ be an arbitrary element of our simply connnected almost Abelian Lie group.
Now we calculate as follows:
\begin{align*}
    \left(\overline{W}(\overline{X_k}) \right)_p &= \overline{\overline{ \overline{W_p} \left(\overline{X_k} \right)_p }} 
    = \overline{ W_p \left( X_k \right)_p} 
    = \overline{ \left( W_j(p) X_p^j \right) \big( (X_k)_p \big)} 
    = \overline{W_k(p)} \overline{\left( X^j X_k \right)_p} 
    = \overline{W_k(p)} \\
    &= \left( \left(\overline{W_j} \, \overline{X^j}\right) \big( \overline{X_k} \big) \right)_p.
\end{align*}
\end{proof}
\end{remark}


We display a formula for left-invariant tensor fields, but it applies to right-invariant tensor fields similarly.
\begin{corollary}\label{TFinframe}
Let $T$ be a smooth section of $TG^{(m,n)} \otimes \overline{TG^{(p,q)}}$. Then $T$ can be expanded in the left invariant tensor frame derived above as follows:
\begin{equation} \label{TensorExpansion}
 T([v,t]) = 
 T^{i_1,\dots,i_m,k_1,\dots,k_p}_{j_1,\dots,j_n,\ell_1,\dots,\ell_q} X_{i_1} \cdots X_{i_m} \cj X_{k_1}\cdots \cj X_{k_p} X^{j_1} \cdots X^{j_n} \cj X^{\ell_1}\cdots \cj X^{\ell_q}
\end{equation}
for all $[v,t] \in G$, where $\left\{ X_i \right\}_{i=1}^{d+1}$ and $\left\{ X^i \right\}_{i=1}^{d+1}$ are the mutually dual frames of left-invariant holomorphic vector fields and covector fields, respectively, and $\left\{ \cj X_j \right\}_{j=1}^{d+1}$ and $\left\{ \cj X_j\right\}_{j=1}^{d+1}$ are their antiholomorphic counterparts.
\end{corollary}


\begin{proposition}\label{ConstantCoeffTensorFields}
The mixed tensor field $T$ as in \eqref{TensorExpansion} is left-invariant if and only if the coefficient functions $$T^{i_1,\dots,i_m,k_1,\dots,k_p}_{j_1,\dots,j_n,\ell_1,\dots,\ell_q}$$ are all constant.

\begin{proof}
For every $g \in G$, define $\Phi_g : G \to G$ by $\Phi_g(x) = gx$. Abusing notation, we let $\Phi_{g}^*$ denote both the pullback and pushforward of $\Phi_g$. 

Since the frames $\{X_i\}_{i=1}^{d+1}$, $\{X^i\}_{i=1}^{d+1}$, $\{\cj X_j\}_{j=1}^{d+1}$, and $\{\cj X^j\}_{j=1}^{d+1}$ are all left-invariant, we know that $\Phi_{g}^* X_j = X_j$, $\Phi_g^* X^j = X^j$, $\Phi_g^* \cj X_j = \cj X_j$, and $\Phi_{g}^* \cj X^j = \cj X^j$ for each $1 \leq j \leq d+1$ and for all $g \in G$. Then we may say that $T$ is left-invariant if and only if
\begin{align*}
     \Phi_{g}^* T &= \Phi_g^*(T^{i_1,\dots,i_m,k_1,\dots,k_p}_{j_1,\dots,j_n,\ell_1,\dots,\ell_q}) \Phi_{g}^* ( X_{i_1} \cdots X_{i_m} \cj X_{k_1}\cdots \cj X_{k_p} X^{j_1} \cdots X^{j_n}  \cj X^{\ell_1}\cdots \cj X^{\ell_q}) \\
     &= \Phi_g^*(T^{i_1,\dots,i_m,k_1,\dots,k_p}_{j_1,\dots,j_n,\ell_1,\dots,\ell_q}) X_{i_1} \cdots X_{i_m} \cj X_{k_1}\cdots \cj X_{k_p} X^{j_1} \cdots X^{j_n}  \cj X^{\ell_1}\cdots \cj X^{\ell_q} \\
     &=  T^{i_1,\dots,i_m,k_1,\dots,k_p}_{j_1,\dots,j_n,\ell_1,\dots,\ell_q} X_{i_1} \cdots X_{i_m} \cj X_{k_1}\cdots \cj X_{k_p} X^{j_1} \cdots X^{j_n}  \cj X^{\ell_1}\cdots \cj X^{\ell_q} = T,
\end{align*}
i.e., iff the scalar functions $$T^{i_1,\dots,i_m,k_1,\dots,k_p}_{j_1,\dots,j_n,\ell_1,\dots,\ell_q}$$ are invariant under the left action of $G$. Since this action is transitive, invariance is equivalent to being constant.
\end{proof}
\end{proposition}

A hermitian form $h$ is a non-degenerate sesquilinear form with hermitian symmetry, written locally that is  
$$h = h_{ab}dz^a \otimes d\cj{z}^b,$$
where $h_{ab}$ is a hermitian matrix. Furthermore a hermitian metric (or structure) is a positive-definite hermitian form. In light of Corollary \ref{TFinframe} we may write any hermitian metric $h$ on a simply connected group $G$ as
$$h = h_{ij}X^i \cj{X}^j,$$
where $h_{ij}$ is a positive definite matrix. We are now prepared to classify all (left) invariant hermitian metrics on a simply connected complex almost Abelian groups. 
\begin{corollary}\label{InvariantHermitianMetricsHaveConstantCoeffs}
Let $h = {h}_{jk} X^j\cj{X}^k$ be a Hermitian metric. Then $h$ is left-invariant if and only if $\widehat{h}_{jk}$ is a constant matrix.
\begin{proof}
For every $g \in G$, define $\Phi_g : G \to G$ by $\Phi_g(x) = gx$.

Since the frames $\{X^i\}_{i=1}^{d+1}$, and $\{\cj X^j\}_{j=1}^{d+1}$ are left-invariant, we know that $\Phi_g^* X^j = X^j$, and $\Phi_g^* \cj X^j = \cj X^j$ for each $1 \leq j \leq d+1$ and for all $g \in G$. Then we may say that $h$ is left-invariant if and only if
\begin{align*}
     \Phi_{g}^* h &= \Phi_{g}^* (\widehat{h}_{jk})\Phi_{g}^* (X^j\cj X^k) \\
     &= \Phi_{g}^* (\widehat{h}_{jk}) X^j \cj X^k \\
     &= \widehat{h}_{jk} X^j \cj X^k = h,
\end{align*}
i.e., iff the scalar functions $\widehat{h}_{jk}$ are invariant under the left action of $G$. Since this action is transitive, invariance is equivalent to being constant.
\end{proof}
\end{corollary}

Lastly, the following Lemma is used implicitly (without mention) for the remainder of this paper.
\begin{lemma}
If $h$ is a Hermitian form on $G$, by definition we have $h(Y, \overline{Z}) = \overline{h(Z, \overline{Y})} \foreach Y,Z \in T_p G,\ p \in G$. Then this implies for arbitrary $V, W \in T_p^{(1,1)} G$ that $h(V,W) = \overline{h(\overline{W}, \overline{V})}$. 
\begin{proof}
Let $V = V_1 + V_2$ and $W = W_1 + W_2$, where $V_1, W_1 \in T_p G$ and $V_2, W_2 \in \overline{T_p G}$. Then we may express:
\begin{align*}
    h(V, W) &= \sum_{i,j \in \{1,2\}} h(V_i, W_j) \\
    \overline{h(\overline{W}, \overline{V})} &= \sum_{i,j \in \{1,2\}} \overline{h(\overline{W_j}, \overline{V_i})}
\end{align*}
But then note that:
\begin{align*}
    h(V_1, W_1) &= 0 = \overline{h(\overline{W_1}, \overline{V_1})}\\
    h(V_1, W_2) &= \overline{h(\overline{W_2}, \overline{V_1})}\\
    h(V_2, W_1) &= 0 = \overline{h(\overline{W_1}, \overline{V_2})}\\
    h(V_2, W_2) &= 0 = \overline{h(\overline{W_2}, \overline{V_2})}
\end{align*}
Thus indeed
\begin{equation*}
    h(V,W) = \overline{h(\overline{W}, \overline{V})}.
\end{equation*}
\end{proof}
\end{lemma}

\section{Nonexistence of Invariant K\"{a}hler Metrics}\label{B3}
We first recall some definitions. 
\begin{definition}[Def. 3.1.1 in \cite{huybrechts2005complex}, pg 114]\label{Kahler}
We define the \textit{fundamental form} of a Hermitian metric $h = h_{ab} dz^a \otimes d\cj z^b$ to be 
\begin{equation*}
    \omega \coloneqq \frac{i}{2} h_{ab} dz^a \wedge d\cj z^b
\end{equation*}
\end{definition}
\begin{definition}[Def 3.1.6 in \cite{huybrechts2005complex}] A \textit{K\"{a}hler structure} is a hermitian structure $h$ for which the fundamental form $\omega$ is closed, i.e. $d\omega = 0$. In this case the fundamental $\omega$ form is called the K\"{a}hler form. 
\end{definition}
We now state the fundamental form of a left-invariant Hermitian metric in terms of our left-invariant frame. 
\begin{lemma} \label{omegaDer}
Let $h = \widehat{h}_{ab}X^a\cj X^b$ be a left-invariant Hermitian metric on a simply connected almost Abelian Lie group $G$. Then the fundamental form of $h$ will be
$$\omega = \frac{i}{2}\widehat{h}_{ab}X^a \wedge \cj X^b.$$
\end{lemma}
\begin{proof}
    From definition \ref{Kahler} we have $\frac{i}{2}h_{ab}dz^a\wedge d\cj z^b$. Recall (see \S3.7 in \cite{tu2007introduction}) that for one forms, say $dz^a$ and $d\cj z^b$, the wedge product is
\begin{align*}
    (dz^a \wedge d\cj z^b)(V_1, V_2) &= \sum_{\sigma \in S_{2}} \sgn(\sigma) dz^a(V_{\sigma(1)}) d\cj z^b(V_{\sigma(2)}) \\
    &= dz^a(V_1) d\cj z^b(V_2) - d\cj z^b(V_1) dz^a(V_2).
\end{align*}
i.e., $dz^a \wedge d\cj{z}^b = dz^a\otimes d\cj z^b-d\cj z^b\otimes dz^a$. We then calculate the following expression for $\omega$,
    \begin{align*}
        \omega &= \frac{i}{2} h_{ab} dz^a\wedge d\cj z^b \\
        &= \frac{i}{2} [h_{ab}dz^{a}\otimes d\cj z^b - h_{ab}d\cj z^b\otimes dz^a]\\
        &= \frac{i}{2} [h_{ab}dz^{a}\otimes d\cj z^b - \overline{\overline{h_{ab}}dz^b\otimes d\cj z^a }]\\
        &= \frac{i}{2} [h_{ab}dz^{a}\otimes d\cj z^b - \overline{h_{ba}dz^b\otimes d\cj z^a }]\\
        &= \frac{i}{2} [h_{ab}dz^{a}\otimes d\cj z^b - \overline{h_{ab}dz^a\otimes d\cj z^b }] \\
        &= \frac{i}{2}[h - \cj{h}].
    \end{align*}
    Note that $h=\widehat{h}_{ab}X^a\cj X^b$, so we further have 
    \begin{align*}
        \omega &= \frac{i}{2}[h - \cj{h}] \\
        &= \frac{i}{2}[\widehat{h}_{ab}X^a\cj X^b-\overline{\widehat{h}_{ab}X^a\cj X^b}] \\ 
        &=  \frac{i}{2}[\widehat{h}_{ab}X^a\cj X^b - \overline{\widehat{h}_{ab}}\cj{X}^a X^b] \\
        &= \frac{i}{2}[\widehat{h}_{ab}X^a\cj X^b-\widehat{h}_{ba}\cj X^a X^b] \\
        &= \frac{i}{2}[\widehat{h}_{ab}X^a\cj X^b-\widehat{h}_{ab}\cj X^b X^a].
    \end{align*}
    Therefore, $\omega = \frac{i}{2}\widehat{h}_{ab}X^a\wedge \cj X^b$ and is the desired form. 
\end{proof}

We now state the main result for this section.
\begin{manualtheorem}{1.1}[Nonexistence of K\"{a}hler Metrics]
There are no left-invariant K\"{a}hler metrics on a simply connected almost Abelian group. 
\end{manualtheorem}
We provide two different proofs of Theorem \ref{NonexistKahlerSimpConn}.


\begin{proof}[Proof I of Theorem \ref{NonexistKahlerSimpConn}]

Let $h$ be an Hermitian structure on a simply connected almost abelian group $G$. We have from Lemma \ref{omegaDer} that 
$$\omega=\frac{i}{2}\widehat{h}_{ab}X^a\wedge \cj X^b,$$
or in other terms: 
\begin{equation*}
    \omega = \widehat{\omega}_{ij}X_a^i \cj X^j_b dz^a \land d\cj z^b,
\end{equation*}
where $\widehat{\omega}_{ij}=\frac{i}{2}\widehat{h}_{ij}$. Define $\widehat{\omega}$ to be the matrix whose $i^{\text{th}}$ row and $j^{\text{th}}$ column is $\widehat{\omega}_{ij}$. Note that from the definition of $\omega$, we see that $\overline{\omega}^\top = \omega$, and so $\overline{\widehat{\omega}}^\top = \widehat{\omega}$. Note also that since $h$ is nondegenerate, $\widehat{h}=(\widehat{h}_{ij})_{1 \leq i,j \leq n}$ is nondegenerate, and thus $\frac{i}{2}\widehat{h}=\widehat{\omega}$ is nondegenerate. If in addition, $\omega$ is closed, i.e. $d\omega=0$, then we call $h$ a Kähler structure.  Assume for contradiction that $h$ is Kähler, and thus $\omega$ is closed.

Define $n=d+1$ to be the dimension of the almost abelian group, hence also the dimension of the matrix 
$$X \coloneqq \begin{pmatrix}
e^{-tJ(\aleph)}&0\\
0& 1
\end{pmatrix}.$$
We name this matrix $X$ since the left-invariant covector field frame is formed by taking the rows of the matrix. Hence $X^i_a$ corresponds to the entry in the $i$-th row and $a$-th column of the matrix $X$.

Then we take the exterior derivative of $\omega$. By Prop. 2.6.15 in \cite{huybrechts2005complex}, we know we may write $d = \partial + \overline{\partial}$ (see pgs. 43-44 of \cite{huybrechts2005complex} for the definitions of $\partial$ and $\overline{\partial}$). Thus we have:
\begin{align*}
    0=d\omega&=(\partial+\cj \partial)\omega\\
    &=\partial (\widehat{\omega}_{ij}X^i_a\cj X^j_b dz^a\wedge d\cj z^b)+\cj\partial (\widehat{\omega}_{ij}X^i_a\cj X^j_b dz^a\wedge d\cj z^b)\\
    &=\widehat{\omega}_{ij}\partial(X^i_a \cj X^j_b dz^a\wedge d\cj z^b)+\widehat{\omega}_{ij}\cj\partial(X^i_a \cj X^j_b dz^a\wedge d\cj z^b)\\
    &=\widehat{\omega}_{ij}(\cj X^j_b \pdv{z^\ell}X_a^i+X^i_a\pdv{z^\ell}\cj X^j_b)dz^\ell\wedge dz^a\wedge d\cj z^b\\
    &\qquad +\widehat{\omega}_{ij}(\cj X^j_b \pdv{\cj z^\ell}X^i_a+X^i_a\pdv{\cj z^\ell}\cj X^j_b)d\cj z^\ell \wedge dz^a\wedge d\cj z^b.
\end{align*}
Now we set
\begin{align*}
    I:&=\widehat{\omega}_{ij}(\cj X^j_b \pdv{z^\ell}X_a^i+X^i_a\pdv{z^\ell}\cj X^j_b)dz^\ell\wedge dz^a\wedge d\cj z^b\\
    II:&=\widehat{\omega}_{ij}(\cj X^j_b \pdv{\cj z^\ell}X^i_a+X^i_a\pdv{\cj z^\ell}\cj X^j_b)d\cj z^\ell \wedge dz^a\wedge d\cj z^b.
\end{align*}

We can further simplify these expressions. We first simplify the $I$ in the equation above. Notice that the only possible variable in any entry of $X$ is $t = z^n$, thus for all $\ell < n$ we have $\frac{\partial}{\partial z^\ell}X^i_a=0$. Thus $I$ becomes
$$
\widehat{\omega}_{ij}(\cj X^j_b \pdv{z^n}X^i_a+X^i_a\pdv{z^n}\cj X^j_b)dz^n\wedge dz^a\wedge d\cj z^b.
$$

Further, since $\cj X^j_b$ is antiholomorphic for all $j,b$, we have that $\frac{\partial}{\partial z^n}\cj X^j_b=0$. Thus the expression above further simplifies to
$$
\widehat{\omega}_{ij}\cj X^j_b \pdv{z^n}X^i_a dz^n\wedge dz^a\wedge dz^b.
$$
Now we simplify $II$. By similar argument, one can show that $II$ simplifies to
$$
\widehat{\omega}_{ij}\pdv{\cj z^n}\cj X^j_b d\cj z^n\wedge dz^a\wedge dz^b.
$$
Put all these together, we arrive at $d \omega=0$ is equivalent to
\begin{equation}\label{SystemOfPDEs}
    \widehat{\omega}_{ij}\cj X^j_b \pdv{z^n}X^i_a dz^n\wedge dz^a\wedge dz^b+\widehat{\omega}_{ij}\pdv{\cj z^n}\cj X^j_b d\cj z^n\wedge dz^a\wedge dz^b = 0.
\end{equation}
This is our systems of PDEs. Now let us discuss by cases.
\begin{enumerate}[label=(\roman*)]
    \item When $a = n$, the first term in \eqref{SystemOfPDEs} vanishes, implying the subsystem of PDEs:
    \begin{equation} \label{CaseiPDEs}
        \widehat{\omega}_{ij}X^i_a \pdv{\cj z^n}\cj X^j_b=0.
    \end{equation}

    \item When $b = n$, the second term in \eqref{SystemOfPDEs} vanishes, implying the subsystem of PDEs:
    \begin{equation} \label{CaseiiPDEs}
        \widehat{\omega}_{ij}\cj X^j_b \pdv{z^n}X^i_a=0.
    \end{equation}

    \item  When $a \neq n$ and $b \neq n$, then since all possible wedge product of 3 1-forms form a basis of the space $\Omega^3(G_\mathbb{R})$, they are linearly independent. Hence each of the coefficient function in \eqref{SystemOfPDEs} must equal zero. Therefore we get the following subsystem of \eqref{SystemOfPDEs}:
    \begin{align} 
        \widehat{\omega}_{ij}\cj X^j_b \pdv{z^n}X^i_a &=0\label{CaseiiiaPDEs}\\
        \widehat{\omega}_{ij}X^i_a\pdv{\cj z^n}\cj X^j_b&=0.\label{CaseiiibPDEs}
    \end{align}
\end{enumerate}


We would like to turn the results above (still case by case) into matrix equations. We define the quantity
$$
\Gamma_{ab}=\widehat{\omega}_{ij}X^i_a\cj X^j_b.
$$
In case (iii) above, we have \eqref{CaseiiiaPDEs}, \eqref{CaseiiibPDEs} as our systems of PDEs. For \eqref{CaseiiiaPDEs}, we can slightly modify it:
$$
\widehat{\omega}_{ij}\cj X^j_b \pdv{z^n}X^i_a+{\widehat{\omega}_{ij}X_a^i\pdv{z^n}\cj X^j_b}=
\frac{\partial}{\partial z^n}\left(\widehat{\omega}_{ij}X^i_a\cj X^j_b\right)=\frac{\partial}{\partial z^n}\Gamma_{ab}=0,
$$
where from the first equality to the second, the first term is zero by \eqref{CaseiiiaPDEs}, and the second term is zero by anti-holomorphicity. Therefore, equation \eqref{CaseiiiaPDEs} is equivalent to
\begin{equation}\label{GammaHoloPDE}
    \frac{\partial}{\partial z^n}\Gamma_{ab}=0.
\end{equation}
Performing similar operations, \eqref{CaseiiibPDEs} can be modified into
\begin{equation}\label{GammaAntiHoloPDE}
    \frac{\partial}{\partial\cj z^n}\Gamma_{ab}=0.
\end{equation}



\noindent We try to express the matrix $\Gamma$ whose entries are $\Gamma_{ab}$ using $X$ and $\widehat{\omega}$ whose entries are $\widehat{\omega}_{ij}$. We have
$$
\Gamma_{ab}=\widehat{\omega}_{ij}X^i_a\cj X^j_b=\widehat{\omega}_{ij}\cj X^j_b X^i_a,
$$
where the second equality is commuting two complex numbers. The operation $\sum_j\widehat{\omega}_{ij}\cj X^j_b$ is clearly the entry corresponding to the $i$-th row and $b$-th column of the matrix $\widehat{\omega}\cj{X}$. Hence we define the matrix $\Xi \coloneqq \widehat{\omega}\cj{X}$. Then
$$
\widehat{\omega}_{ij}\cj X^j_b=\Xi_{ib}.
$$
Then we have
$$
\Gamma_{ab}=\Xi_{ib}X^i_a=(\Xi^\top X)_{ba}=(\cj{X}^\top \widehat{\omega}^\top X)_{ba}.
$$
Transposing the entire equation once again we find
$$\Gamma_{ab}=(X^\top \widehat{\omega}\cj{X})_{ab}.$$ 
Or equivalently,
\begin{equation}\label{matrixGamma}
    \Gamma = X^\top \widehat{\omega} \cj{X}.
\end{equation}
Recall that $\cj{\omega}^\top = \omega$. Then from \eqref{matrixGamma}, it follows that $\cj{\Gamma}^\top = \Gamma$. Substitute (\ref{matrixGamma}) into the cases ((i)-(iii)) from the previous section, whose systems of PDEs are translated into \eqref{GammaHoloPDE} and \eqref{GammaAntiHoloPDE}. Then the work translates:
\begin{enumerate}[label=(\roman*)]
    \item When $a = n$, we have that $\Gamma_{nb}$ must satisfy the $nb$ component of the following equation:
    \begin{equation} \label{GammaPDEi}
        X^\top \widehat{\omega} \left( \overline{-J(\aleph)} \oplus 0 \right) \cj{X} = 0.
    \end{equation}
    
    \item When $b=n$, we have that $\Gamma_{an}$ must satisfy the $an$ component of the following equation:
    \begin{equation} \label{GammaPDEii}
        (-J(\aleph) \oplus 0)^\top X^\top \widehat{\omega} \cj{X} = 0.
    \end{equation}
    
    \item When $a \neq b$ and $b \neq n$, we have that $\Gamma_{ab}$ must satisfy the $ab$ component of the following equations:
    \begin{align}
        (-J(\aleph) \oplus 0)^\top X^\top \widehat{\omega} \cj{X} &= 0 \label{GammaPDEiiia}\\
        X^\top \widehat{\omega} \left( \overline{-J(\aleph)} \oplus 0 \right) \cj{X} &= 0. \label{GammaPDEiiib}
    \end{align}
\end{enumerate}

Thus we are searching for $\widehat{\omega}$ satisfying \eqref{GammaPDEi} only when $a=n$, \eqref{GammaPDEii} only when $b=n$, and satisfying \eqref{GammaPDEiiia} and \eqref{GammaPDEiiib} simultaneously for all other $(a,b)$. Note that \eqref{GammaPDEi} is the same equation as \eqref{GammaPDEiiib}, and \eqref{GammaPDEii} is the same equation as \eqref{GammaPDEiiia}. Thus going forward, we only refer to \eqref{GammaPDEiiia} and \eqref{GammaPDEiiib}.

Now note that \eqref{GammaPDEiiia} is equivalent to:
\begin{equation} \label{GammaiiiaEq}
    \left( -J(\aleph) \oplus 0 \right)^\top \Gamma = 0.
\end{equation}
Similarly, since $\cj{X} = e^{-t\overline{J(\aleph)}} \oplus 1$ and the exponential is in particular a series in $\overline{J(\aleph)}$, it follows that $\overline{-J(\aleph)} \oplus 0 $ commutes with $\overline{X}$. Thus \eqref{GammaPDEiiib} is equivalent to:
\begin{equation} \label{GammaiiibEq}
    \Gamma \left( \overline{-J(\aleph)} \oplus 0 \right) = 0.
\end{equation}
But then we may take the conjugate transpose of \eqref{GammaiiiaEq}, which yields \eqref{GammaiiibEq}, and thus \eqref{GammaPDEiiia} is equivalent to \eqref{GammaPDEiiib}. This is very nice because it means we may pick one of these two equivalent equations, say \eqref{GammaPDEiiia}, and it is the condition that all three of our cases ((i) - (iii)) must satisfy. Thus our entire system of PDEs reduces to the singular matrix equation (and no cases), which is \eqref{GammaPDEiiia}.

Like before, observe that $X = e^{-tJ(\aleph)} \oplus 1$ implies that $\left(-J(\aleph) \oplus 0 \right)^\top$ commutes with $X^\top$. Thus we have that \eqref{GammaPDEiiia} is equivalent to:
\begin{align}
    0 &= (-J(\aleph) \oplus 0)^\top X^\top \widehat{\omega} \cj{X} \\
    &= X^\top (-J(\aleph) \oplus 0)^\top \widehat{\omega} \cj{X}. \label{sec} 
\end{align}
Now, use the fact that $X$ is invertible to multiply out $X^\top$ and $\cj{X}$ \eqref{sec} to get:
\begin{equation} \label{ContradictionEq}
    (-J(\aleph) \oplus 0)^\top \widehat{\omega} = 0.
\end{equation}

But then recall that $\widehat{\omega}$ is non-degenerate as a consequence of its definition in terms of the Hermitian form $h$. Thus we may multiply both sides of \eqref{ContradictionEq} by $\widehat{\omega}\inv$ to get 
\begin{equation}
    (-J(\aleph) \oplus 0)^\top = 0,
\end{equation}
which is a contradiction unless $J(\aleph) = 0$. But then if $J(\aleph) = 0$, we would have that $G$ is actually Abelian, a contradiction to the almost Abelian condition.
\end{proof}


\begin{proof}[Proof II of Theorem \ref{NonexistKahlerSimpConn}]
Here we present an alternative proof by contradiction to show that there are no K\"{a}hler structures on complex simply connected almost Abelian Lie groups. Let $G$ be a simply connected complex almost Abelian group of dimension $d+1$ and $\mathfrak{g}$ be it's Lie algebra. Left-invariant vector fields on $G$ can be identified with tangent vectors at the identity, forming the Lie algebra $\mathfrak{g}$.

Let $h$ be a left-invariant Hermitian structure on a simply connected almost Abelian group $G$. We have by Lemma \ref{omegaDer}
$$\omega = \frac{i}{2}\widehat{h}_{ij}X^i\wedge \cj X^j.$$
From here we identify $\omega$ with the underlying real structure of $G$. We take the frame to be $\{X_1, \dots, X_{d+1}, \cj X_1, \dots, \cj X_{d+1} \}$ 
Where $\{X_i\}^{d+1}_{i=1}$ and $\{\cj X_i\}^{d+1}_{i=1}$ are the basis of our left-invariant holomorphic and antiholomorphic frames respectively, here when evaluated at some point $p \in M$ are thought of as real tangent vectors in $T_pM_{\R}$. We introduce some notation for our frame, $i,j,k$ refer specifically to the index of the holomorphic or antiholomorphic frame elements and we use $r, s, t$ as running indices through our entire frame, that is:
$$X_s := \begin{cases} X_s &\quad \mbox{if } s \leq d+1 \\
\cj{X}_ {s-d-1} &\quad \mbox{if } s > d+1 
\end{cases}.$$
When using indices to refer to matrix elements, we let $\overline{i} := d + 1 + i$ and again use $r, s, t$ as running indices through the entire matrix.
As we have already shown for an invariant Hermitian structure $\widehat{h_{st}}$ must be constant we further have that $\omega$ will also have constant matrix components. We have
$$\omega(X_i, X_j) = 0, \;\; \omega(\cj X_i,\cj X_j) = 0,$$
$$\omega(X_i, \cj X_j) := \omega_{i\overline{j}}, \; \; \omega(\overline{X_i}, X_j) := \omega_{\overline{i}j}.$$
Where $\omega_{i\overline{j}}, \omega_{\overline{i}j}$ are constant. By this notational convention, $\omega_{\overline{i}\overline{j}} = \omega_{\overline{i}\overline{j}} = 0$.
As $\omega$ is a 2-form, we may employ proposition 4.1.6 of \cite{RuSc13}, page 170 to see that the formula for the exterior derivative simplifies to
\begin{align}
 (d\omega)(X_r, X_s, X_t) &= X_r(\omega(X_s, X_t)) - X_s(\omega(X_r, X_t)) + X_t(\omega(X_r, X_s)) \nonumber \\ 
 & \qquad \qquad \quad - \omega([X_r, X_s], X_t) + \omega([X_r, X_t], X_s) - \omega([X_s, X_t], X_r) \nonumber \\ 
 &= - \omega([X_r, X_s], X_t) + \omega([X_r, X_t], X_s) - \omega([X_s, X_t], X_r), \label{simpD}
\end{align}
as $\omega(X_s, X_t)$ are all constants. Recall that $[X_i, \cj X_j] = 0$, that $[X_i, X_j]$ will be holomorphic, and $[\cj X_i,\cj X_j]$ will be anti-holomorphic, i.e. $\omega([X_i, X_j], X_k) = 0$ and $\omega([\cj X_i,\cj X_j], \cj X_k) = 0$. The exterior derivative then further simplifies depending on how many holomorphic or anti-holomorphic frame elements are being evaluated.
\begin{align*}
 (d\omega)(X_i, X_j, X_k) &= - \omega([X_i, X_j], X_k) + \omega([X_i, X_k], X_j) - \omega([X_j, X_k], X_i) = 0, \\
 (d\omega)(\cj X_i,\cj X_j,\cj X_k)&=-\omega([\cj X_i,\cj X_j],\cj X_k)+\omega([\cj X_i,\cj X_k],\cj X_j)-\omega([\cj X_j,\cj X_k],\cj X_i)=0, \\
 (d\omega)(\cj X_i, X_j, X_k) &= - \omega([\cj X_i, X_j], X_k) + \omega([\cj X_i, X_k], X_j) - \omega([X_j, X_k], \cj X_i) = - \omega([X_j, X_k], \cj X_i),\\
 (d\omega)(\cj X_i,\cj X_j,X_k)&=-\omega([\cj X_i,\cj X_j],X_k)+\omega([\cj X_i,X_k],\cj X_j)-\omega([\cj X_j,X_k],\cj X_i)=-\omega([\cj X_i,\cj X_j],X_k).
\end{align*}
Recall that we may permute the entries of $d\omega$ so this indeed covers all cases. By way of contradiction we suppose $d\omega = 0$. Thus, 
\begin{align}
     \omega([X_j, X_k], \cj X_i) &= 0, \label{2holo}\\
     \omega([\cj X_i, \cj X_j], X_k) &= 0. \label{2anti}
\end{align}

We pick the basis $\{V_1, \dots, V_d, V_{d+1}\}$ for the Lie algebra $\mathfrak{g}$ such that $\{V_1, \dots, V_d\}$ generate the Abelian ideal, i.e. $V_{d+1} = e_0$. We further identify this with a basis for the real Lie algebra $\mathfrak{g_{\R}}$ as follows 
$$\{V_1, \dots, V_d, V_{d+1}, \cj V_1,...,\cj V_d,\cj V_{d+1}\}.$$
We have the simplification that, 
\begin{align*}
    [V_i, V_j] &= 0 \qquad 1 \leq i, j \leq d \\
    [\cj V_i,\cj V_j] &= 0 \qquad 1 \leq i, j \leq d
\end{align*}
So only terms involving $[e_0, V_i]$ and $[\cj{e}_0, \cj V_i]$ are left in the expressions \eqref{2holo} \eqref{2anti} respectively. Namely, it simplifies to 
\begin{align}
    \omega([e_0, V_i], \cj V_j) &= 0, \qquad 1 \leq i, j \leq d+1, \label{enot}\\
     \omega([\cj e_0, \cj V_i], V_j) &= 0, \qquad 1 \leq i, j \leq d+1, 
\end{align}
Recall that we already have that 
\begin{align*}
    \omega([e_0, V_i], V_j) = \omega([e_0, \cj V_i], V_j) = \omega([e_0, \cj V_i], \cj V_j) =0, \qquad 1 \leq i,j \leq d+1 . 
\end{align*}
As $[X,Y] = \ad_X Y$ the condition from \eqref{enot} can be written 
$$\omega(\ad_{e_0}(V_s), (V_t) ) = 0, \quad 1 \leq s, t \leq 2d+2.$$ 
So we have that $\omega$ an anti-symmetric bilinear form such that
\begin{equation} \label{adj}
    \omega(\ad_{e_0}(-), (-) ) = 0.
\end{equation} 
Now take some $V_s \in \{V_1, \dots, V_{d+1}, \cj V_1,...,\cj V_{d+1}\}$, we have for all $V_t \in \{V_1, \dots, V_{d+1}, \cj V_1,...,\cj V_{d+1}\}$:
$$\omega(\ad_{e_0}(V_s), V_t) = 0.$$
Which implies that $\ad_{e_0}(V_s) = 0$ as $\omega$ is non-degenerate. Thus we have that for all $s \in [1, 2d+2]$, $\ad_{e_0}(V_s) = 0$. Recalling that $\ad_{e_0}$ completely determines the Lie Algebra structure as shown in \cite{Ave18}, we have that $G$ is Abelian which is a contradiction. 
\end{proof}


Analogously we have for right-invariant K\"{a}hler metrics:

\begin{theorem}
There do not exist any right-invariant K\"{a}hler metrics on simply connected complex almost Abelian groups. 
\end{theorem}

\begin{proof}
We may also identify right-invariant vector fields on $G$ with tangent vectors at the identity, which form the Lie algebra $\mathfrak{g}$. Utalizing the right invariant tensor and vector field frames found in Corollary \ref{frameOfHoloVF} and Remark \ref{frameOfAntiVF}, the above proof may be applied to the case of right-invariant Hermitian metrics similarly.
\end{proof}

\;


\section{Generalization to Connected Groups}\label{B4}

We use the fact that every connected Lie group $G$ is the quotient of a unique simply connected Lie group $\widetilde{G}$ (having the same Lie algebra) by a discrete subgroup $\Gamma$ in order to generalize our results on Hermitian and K\"{a}hler structures to connected complex almost Abelian groups. Namely, we know that the covering map is a quotient map by a discrete subgroup $\Gamma$.

\begin{proposition} \label{DiscQuotientHerm}
Let $\Gamma \subseteq \widetilde{G}$ be a discrete normal (equiv. central) subgroup. Let $q_\Gamma : \widetilde{G} \to G \coloneqq \widetilde{G} / \Gamma$ be the standard quotient map. Then the pullback of any Hermitian metric $h$ on $G$ is a  right-$\Gamma$-invariant Hermitian metric $\widetilde{h}$ on $\widetilde{G}$, and locally the pullback under $q_{\Gamma}\inv$ of any right-$\Gamma$-invariant Hermitian metric $\widetilde{h}$ on $\widetilde{G}$ yields a Hermitian metric on $G$.

\begin{proof}
Since $q_\Gamma$ is a smooth covering map, it is a local diffeomorphism. Thus $dq_\Gamma$ is an isomorphism between tangent spaces at each point, i.e., an invertible linear transformation. 
Let $h$ be an arbitrary Hermitian structure on $G$.
Then define $\widetilde{h}(Y,Z) \coloneqq q_{\Gamma}^* h$. Then
\begin{align*}
    \overline{\widetilde{h}(Y,Z)} &= \overline{h((q_{\Gamma})_* Y, (q_{\Gamma})_* Z)} \\
    &= h((q_{\Gamma})_* Z, (q_{\Gamma})_* Y) \\
    &= \widetilde{h}(Z,Y).
\end{align*}
Thus $\widetilde{h}$ has Hermitian symmetry at each point. Note that $\widetilde{h} = q_\Gamma^* h$ implies that $\widetilde{h}$ is smooth.

We denote elements of $G$ as equivalence classes $[p]$ of elements $p \in \widetilde{G}$. Now suppose there exists a vector $Y \in T_p^{(1,1)} \widetilde{G}$ such that $\widetilde{h}(Y,Z) = 0$ for each $Z \in T_p^{(1,1)} \widetilde{G}$. Then we must have $h((q_{\Gamma})_* Y, W) = 0$ for all $W \in T_{[p]}^{(1,1)} \widetilde{G}$ (because $(q_{\Gamma})_* Z = [dq_{\Gamma}(p)] Z$, and $[dq_{\Gamma}(p)]$ was established to be invertible for all $p \in \widetilde{G}$). Thus $(q_{\Gamma})_* Y = 0$ by the non-degeneracy of $h$. But then since once again $(q_{\Gamma})_* Y = [dq_{\Gamma}(p)] Y$ and $[dq_{\Gamma}(p)]$ is invertible, we  must have $Y = 0$. Thus $\widetilde{h}$ is non-degenerate. It is apparent $\widetilde{h}$ is positive by definition, and thus it is indeed positive definite. Thus $\widetilde{h}$ is a Hermitian metric, and so the pullback of any Hermitian metric on $G$ yields a Hermitian metric on $\widetilde{G}$.

For all right-$\Gamma$-invariant Hermitian metrics $\widetilde{h}$ on $\widetilde{G}$, define
\begin{equation*}
    h(Y_{[p]},Z_{[p]}) \coloneqq \widetilde{h}\left( [dq_\Gamma(p)]\inv Y_{[p]}, [dq_\Gamma(p)]\inv Z_{[p]} \right),
\end{equation*}
where it is understood that $[dq_\Gamma(p)]\inv : T_{[p]}^{(1,1)} \widetilde{G} \longrightarrow T_p^{(1,1)} \widetilde{G}$.
Now note that by the equivariance of $q_\Gamma$ with respect to the action of right multiplication by $\gamma \in \Gamma$, we have:
\begin{align*}
    h(Y_{[p\gamma]},Z_{[p \gamma]}) &= \widetilde{h}\left( [dq_\Gamma(p \gamma)]\inv Y_{p \gamma}, [dq_\Gamma(p \gamma)]\inv Z_{p\gamma} \right) \\
    &= \widetilde{h}\left( [dq_\Gamma(p)]\inv Y_{[p]}, [dq_\Gamma(p)]\inv Z_{[p]} \right) \\
    &= h(Y_{[p]}, Z_{[p]}).
\end{align*}
Therefore, $h$ is indeed well-defined. Then
\begin{align*}
    \overline{h(Y_{[p]},Z_{[p]})} &= \overline{\widetilde{h}([dq_\Gamma(p)]\inv Y_{[p]}, [dq_\Gamma(p)]\inv Z_{[p]})}\\ 
    &= \widetilde{h}([dq_\Gamma(p)]\inv Z_{[p]}, [dq_\Gamma(p)]\inv Y_{[p]})\\
    &= h(Z_{[p]}, Y_{[p]}).
\end{align*}
Thus $h$ has Hermitian symmetry at each point. Furthermore, $h = \widetilde{h} \circ ([dq_\Gamma(p)]\inv, [dq_\Gamma(p)]\inv)$ is smooth because it is the composition of smooth functions. 

Now suppose there exists $Y_{[p]} \in T_{[p]}^{(1,1)} \widetilde{G}$ such that $h(Y_{[p]},Z_{[p]}) = 0 \foreach Z_{[p]} \in T_{[p]}^{(1,1)} \widetilde{G}$. Similarly as earlier, we must have $\widetilde{h}([dq_\Gamma(p)]\inv Y_{[p]}, W) = 0 \foreach W \in T_p^{(1,1)} \widetilde{G}$, which occurs if and only if $[dq_\Gamma(p)]\inv Y_{[p]} = 0$. But then since $dq_\Gamma$ is invertible, we have $Y_{[p]} = 0$, and so $h$ is non-degenerate at each point. Thus $h$ is a Hermitian metric on $G$.
\end{proof}
\end{proposition}   


\begin{corollary}
Let $\Gamma \subseteq \widetilde{G}$ be a discrete normal (equiv. central) subgroup. Let $q_\Gamma : \widetilde{G} \to G \coloneqq \widetilde{G} / \Gamma$ be the standard quotient map. Then the pullback of any left-$G$-invariant Hermitian metric $h$ on $G$ is a right-$\Gamma$-invariant, left-$\widetilde{G}$-invariant Hermitian metric $\widetilde{h}$ on $\widetilde{G}$, and locally the pullback under $q_{\Gamma}\inv$ of any right-$\Gamma$-invariant, left-$\widetilde{G}$-invariant Hermitian metric $\widetilde{h}$ on $\widetilde{G}$ yields a left-$G$-invariant Hermitian metric on $G$.

\begin{proof}
By Prop. \ref{DiscQuotientHerm}, all that there is to check is that the metric obtained on $G$ by a left-$\widetilde{G}$-invariant, right-$\Gamma$-invariant Hermitian metric $\widetilde{h}$ on $\widetilde{G}$ down to a left-invariant metric on $G$. But then this is apparent from the fact that there is a left action of $\widetilde{G}$ on the quotient by $\Gamma$, namely right $\Gamma$ multiplication on left coset representatives, and that $q_\Gamma$ is clearly equivariant with respect to this action.
\end{proof}
\end{corollary}


Similarly for invariant K\"{a}hler metrics we have:

\begin{proposition}
Let $\Gamma \subseteq \widetilde{G}$ be a discrete normal (equiv. central) subgroup of the simply connected group $\widetilde{G}$. Let $q_\Gamma : \widetilde{G} \to G \coloneqq \widetilde{G} / \Gamma$ be the standard quotient map. Then all left-$G$-invariant K\"{a}hler metrics on $G$ pullback to left-$\widetilde{G}$-invariant K\"{a}hler metrics on $\widetilde{G}$. Consequentially, there are no left-$G$-invariant K\"{a}hler metrics on $G$.
\begin{proof}
Let $h$ be a Hermitian metric on $G$ with corresponding fundamental form $\omega$. Assume $h$ is K\"{a}hler, and so $d\omega = 0$. Now since the pullback commutes with the exterior derivative, we have that $0 = q_{\Gamma}^* 0 = q_{\Gamma}^* d \omega = d(q_{\Gamma}^* \omega)$. 

Let $\widetilde{h} \coloneqq q_{\Gamma}^* h$ be the Hermitian metric on $\widetilde{G}$ induced by the pullback (we know this is indeed a Hermitian metric by Prop. \ref{DiscQuotientHerm}). Let $\widetilde{\omega}$ be the corresponding fundamental form of $\widetilde{h}$. Thus if $q_{\Gamma}^* \omega = \widetilde{\omega}$, then we will have established that $d\widetilde{\omega} = 0$, implying $\widetilde{h}$ is a K\"{a}hler metric on $\widetilde{G}$. 

Written locally, we have $\Tilde{h} = q_{\Gamma}^*h = (h_{ab} \circ q_{\Gamma}) q_{\Gamma}^*(dz^a)\otimes q_{\Gamma}^*(d\cj z^b)$. Therefore the fundamental form $\widetilde{\omega}$ of $\widetilde{h}$ will be 
$$\widetilde{\omega} = \frac{i}{2}(h_{ab} \circ q_{\Gamma})q_{\Gamma}^*(dz^a)\wedge q_{\Gamma}^*(d\cj z^b).$$
Considering the pullback of $\omega$ we have 
\begin{align*}
    q_{\Gamma}^*\omega &= q_{\Gamma}^*(\frac{i}{2} h_{ab}dz^a \wedge d\cj z^b) \\
    &= q_{\Gamma}^*\left(\frac{i}{2}h_{ab} ( dz^a \otimes d\cj z^b - d\cj z^b \otimes dz^a )\right) \\
    &= \left(\frac{i}{2}h_{ab} \circ q_{\Gamma}\right) \left[ q_{\Gamma}^*(dz^a) \otimes q_{\Gamma}^*(d\cj z^b) - q_{\Gamma}^*(d\cj z^b) \otimes q_{\Gamma}^*(dz^a) \right] \\
    &= \frac{i}{2}(h_{ab} \circ q_{\Gamma})q_{\Gamma}^*(dz^a) \wedge q_{\Gamma}^*(d\cj z^b)= \widetilde{\omega}.
\end{align*}
\end{proof}
\end{proposition}

\section{Acknowledgements}
This work was done as part of the University of California, Santa Barbara Mathematics Summer Research Program for Undergraduates and was supported by NSF REU Grant DMS 1850663. We are very grateful to both UCSB and the NSF for making this opportunity possible, and for the enriching, challenging, and fun experiences we had in the course of the program.

\newpage

\printbibliography

\end{document}